\documentclass[a4paper,draft,reqno,12pt]{amsart}
\usepackage[utf8]{inputenc}
\usepackage[T1,T2A]{fontenc}
\usepackage[english]{babel}
\usepackage{amsmath}
\usepackage{amssymb}
\usepackage{amscd}
\usepackage{amsthm}
\usepackage{euscript}
\usepackage[all]{xy}
\usepackage{tikz-cd}
\usepackage{amsfonts}
\usepackage[shortlabels]{enumitem}
\setlist[enumerate]{label=\normalfont{(\arabic*)}}
\usepackage[colorlinks=true, linkcolor=blue, citecolor=blue]{hyperref}
\usepackage{mathrsfs}
\newtheorem{proposition}{Proposition}

\newtheorem{lemma}{Lemma}
\newtheorem{theorem}{Theorem}
\newtheorem{corollary}{Corollary}
\theoremstyle{definition}

\theoremstyle{remark}
\newtheorem {remark}{Remark}

\DeclareMathOperator{\Spec}{Spec}
\DeclareMathOperator{\Aut}{Aut}
\DeclareMathOperator{\SAut}{SAut}
\DeclareMathOperator{\Supp}{Supp}
\DeclareMathOperator{\SL}{SL}
\DeclareMathOperator{\reg}{reg}
\DeclareMathOperator{\Xreg}{X^{\reg}}
\DeclareMathOperator{\Cl}{Cl}

\DeclareMathOperator{\Susp}{Susp}

\DeclareMathOperator{\divv}{div}
\DeclareMathOperator{\red}{red}
\DeclareMathOperator{\ssss}{ss}
\DeclareMathOperator{\Ker}{Ker}
\DeclareMathOperator{\Frac}{Frac}
\DeclareMathOperator{\WDiv}{WDiv}
\renewcommand{\Im}{\mathop{\text{Im}}}

\DeclareMathOperator{\fdiv}{div}
\renewcommand{\div}{\fdiv}
\DeclareMathOperator{\mIm}{Im}
\renewcommand{\Im}{\mIm}
\newcommand{\bb}{\mathbb}
\renewcommand{\phi}{\varphi}

\def\GG{{\mathbb G}}

\def\CC{{\mathbb C}}
\def\KK{{\mathbb K}}

\def\ZZ{{\mathbb Z}}

\def\AA{{\mathbb A}}

\begin{document}


\date{}
\title{On flexibility of affine factorial varieties}
\author{Ivan Arzhantsev}
\address{Faculty of Computer Science, HSE University, Pokrovsky Boulevard 11, Moscow, 109028 Russia}
\email{arjantsev@hse.ru}
\author{Kirill Shakhmatov}
\address{Faculty of Computer Science, HSE University, Pokrovsky Boulevard 11, Moscow, 109028 Russia}
\email{kshahmatov@hse.ru}
\thanks{The article was prepared within the framework of the project ``International Academic Cooperation'' HSE University}

\maketitle


\begin{abstract}
We give a criterion of factoriality of a suspension. This allows to construct many examples of flexible affine factorial varieties. In particular, we find
a homogeneous affine factorial 3-fold that is not a homogeneous space of an algebraic group. 
\end{abstract} 


\section*{Introduction}

We work over an algebraically closed field~$\KK$ of characteristic zero. Let us say that an algebraic variety $X$ is homogeneous if the automorphism group $\Aut(X)$ acts on $X$ transitively. Examples of homogeneous varieties are homogeneous spaces of algebraic groups. Namely, if $G$ is an algebraic group and $H$ is a closed subgroup of $G$, then the variety $X=G/H$ of left cosets is obviously a homogeneous variety.

Since the group $\Aut(X)$ need not be algebraic, one may expect that the class of homogeneous varieties is wider than the class of homogeneous spaces. This is indeed the case. For example, any smooth quasi-affine toric variety is a homogeneous variety; see ~\cite[Theorem~0.2(2)]{AKZ} and~\cite[Theorem~4.3(a)]{AShZ}. An explicit example of a smooth quasi-affine toric surface that is not a homogeneous space can be found in~\cite[Example~2.2]{AKZ}.

In~\cite{AZa}, we study homogeneous affine varieties. The study is based on the concept of a flexible variety and on the construction of a suspension. Let $\GG_\mathrm{a} $ be the additive group of the ground field~$\KK$ and $\GG_\mathrm{a} \times X \to X$ be a regular action on an algebraic variety $X$. The corresponding subgroup of the automorphism group $\Aut(X)$ is called a $\GG_\mathrm{a}$-subgroup. The special automorphism group $\SAut(X)$ is the subgroup of $\Aut(X)$ generated by all $\GG_\mathrm{a}$-subgroups.

Denote by $\Xreg$ the set of smooth points in~$X$. A point $x\in \Xreg$ is called flexible if the tangent space $T_xX$ is spanned by tangents to orbits of $\GG_\mathrm{a}$-subgroups passing through the point~$x$. A variety $X$ is called flexible if any point $x \in \Xreg$ is flexible.

In \cite[Theorem 0.1]{AFKKZ} it is proved that an irreducible affine variety $X$ is flexible if and only if the group $\SAut(X)$ acts on $\Xreg$ transitively. Moreover, if $X$ has dimension at least~$2$, then these conditions are equivalent to the condition that $\SAut(X)$ acts on $\Xreg$ infinitely transitively. Recall that an action of a group $G$ on a set $X$ is infinitely transitive if it is $m$-transitive for any positive integer $m$, i.e., for any two tuples $(x_1, \ldots, x_m)$ and $(y_1, \ldots, y_m)$ of pairwise distinct points of $X$ there exists an element $g \in G$ such that $gx_i = y_i$ for all~$1 \le i \le m$.

Now we recall the definition of a suspension. Let $Y$ be an affine variety and $f \in \KK[Y]$ be a non-constant regular function. The hypersurface $\Susp(Y, f)$ given in the direct product $\AA^2 \times Y$ by the equation $uv = f$, where $\AA^2 = \Spec \KK[u,v]$, is called a suspension over $Y$.  In the context of automorphism groups suspensions were considered for the first time in~\cite{KZ}. In particular, it is proved in~\cite{KZ} that any suspension over an affine space is flexible. In~\cite[Theorem 0.2(3)]{AKZ} we show that $\Susp(Y, f)$ is flexible provided the variety $Y$ is flexible.

Another reason to study suspensions is topological considerations in the case of ground field $\KK = \CC$. For example, one can compute the homology groups of a suspension $X =$ \linebreak
$= \Susp(\CC^n, f)$ using the homology groups of the hypersurface $\{ f = 0 \}$ in $\CC^n$. In particular, if $X$ is smooth and $\{ f = 0 \}$ is acyclic, then $X$ is contractible \cite[Corollary~4.1]{KZ}. Theoretically, these ideas may be used to provide a counterexample to Zariski Cancellation Problem, or to construct an exotic algebraic structure on $\CC^n$. See~\cite[Section~4]{KZ} for details.

In~\cite[Corollary~1]{AZa} it is proved that $\Susp(Y, f)$ is smooth if and only if both the variety~$Y$ and the scheme $\Spec\KK[Y] / (f)$ are smooth. Since any smooth flexible variety is homogeneous, we obtain many examples of homogeneous affine varieties. Moreover, it is shown in~\cite[Lemma~2]{AZa}  that the Picard number of an affine homogeneous space $X$ is at most $\dim X$.  Since we can construct smooth suspensions over an affine space with the Picard number bigger than the dimension of the suspension, we obtain homogeneous affine varieties that are not homogeneous spaces; see~\cite[Theorem~4]{AZa}.

In this paper we concentrate on homogeneous affine factorial varieties. Again examples of such varieties can be found among homogeneous spaces.  Recall that any connected linear algebraic group $G$ is a semidirect product $G^{\red} \ltimes R_u(G)$ of a reductive group $G^{\red}$ and the unipotent radical $R_u(G)$, see e.g.~\cite[Theorem~6.4]{OV}.  If the group $G^{\red}$ is a direct product of a torus $Z$ and a simply connected semisimple group $G^{\ssss}$,  then the group $G$ is factorial~\cite[Theorem~6]{Po}. Moreover, if $H$ is a connected semisimple subgroup of such a group $G$, then the homogeneous space $G/H$ is affine and by~\cite[Theorem~4]{Po} the variety $G/H$ is factorial.

Our aim is to construct many flexible affine factorial varieties via suspensions. The paper is organized as follows. We begin with some preparatory lemmas in Section~\ref{sec1}. They are used in Section~\ref{sec2} to prove that the suspension $X=\Susp(Y,f)$ is a factorial variety if and only if $Y$ is factorial and $f$ is a prime element in $\KK[Y]$. This result (Theorem~\ref{tmain}) is obtained in the broadest possible context of arbitrary commutative domains. In these settings the passage from $Y$ to $X=\Susp(Y,f)$  may be considered as a generalization of replacing a ring $R$ with the Laurent polynomial ring $R[T,T^{-1}]$.

Combining this with the criterion of smoothness of suspensions from~\cite[Corollary~1]{AZa}, we observe in Section~\ref{sec3} that for a smooth flexible affine factorial variety $Y$ the suspension $\Susp(Y,f)$ is smooth factorial if and only if the function $f$ is prime in $\KK[Y]$ and $\{f=0\}$ is a smooth subvariety in $Y$. It provides many examples of homogeneous affine factorial varieties. We proceed with some results on topology of suspensions over the field of complex numbers. We prove that if the variety $Y$ is simply connected, then any suspension $\Susp(Y,f)$ is simply connected as well (Proposition~\ref{psc}). It follows from~\cite[Theorem~2 and Proposition~19]{Po-2} that there are precisely three pairwise non-isomorphic affine factorial 3-folds without non-constant invertible functions that are homogeneous spaces of algebraic groups: they are $\CC^3$, $\SL(2)$ and $\SL(2)/2I$, where $2I$ is the binary icosahedral group. We show that the hypersurface $\{x^2y-xy-uv+1=0\}$, which is  the suspension $X=\Susp(\CC^2,(x-1)xy+1)$,  has $\ZZ^2$ as the third homotopy group~$\pi_3(X)$. We conclude that $X$ is a homogeneous affine factorial 3-fold that is not a homogeneous space of an algebraic group (Theorem~\ref{texa}).

It is an important problem to find algebraic invariants that allow us to distinguish between non-isomorphic smooth flexible affine factorial varieties of the same dimension. For example, in order to prove that a suspension $\Susp(\AA^2,f)$ is not isomorphic to $\AA^3$ one may use the positive solution of the Linearization Problem, see~\cite{KR}. It implies that any action of an algebraic torus on $\AA^3$ has a fixed point. So it suffices to find a free action of a one-dimensional torus on $\Susp(\AA^2,f)$. It seems to be a more difficult task to distinguish between some $\Susp(\AA^2,f)$ and $\SL(2)$.

The divisor class group of a suspension $\Susp(Y,f)$ is computed in Section~\ref{sec4}. This gives an alternative proof of Theorem~\ref{tmain} in geometric settings. Finally, in Section~\ref{sec5} we show that the natural properties of the projection from $\Susp(Y,f)$ to $Y$ are not sufficient to characterize suspensions.

The authors thank the referees for useful comments and  suggestions.


\section{Basic lemmas}
\label{sec1}


Let $R$ be a commutative associative ring with unit. Consider the polynomial ring $R[u,v]$ and fix an element $f\in R$. The main object studied in this section is the factor ring
$$
S:=R[u,v]/(uv-f).
$$
This ring carries a natural $\ZZ$-grading $S=\oplus_{i\in\ZZ} S_i$ given by
$$
S_0=R, \quad S_i=Ru^i  \quad \text{and} \quad S_{-i}=Rv^i \quad \text{for}  \ i>0.
$$
Indeed, the relation $uv=f$ implies that $S$ is the sum of subspaces $S_i$. The inclusion ${S_iS_j\subseteq S_{i+j}}$ is clear.  If we have $r_0+\sum_{i>0} r_iu^i+\sum_{j>0} r_j'v^j=0$ in $S$, then
$
r_0+\sum_{i>0} r_iu^i+\sum_{j>0} r_j'v^j=h(u,v)(uv-f)
$
in $R[u,v]$ for some polynomial $h(u,v)$. Consider the lexicographic order $\succ$ on terms in $u$ and $v$ with $u\succ v$. If $h$ is nonzero, then the leading term of $h(u,v)(uv-f)$ is divisible by $uv$, while there is no such term on the left hand side.

\smallskip

The following lemma generalizes~\cite[Lemma~3.1]{AKZ}.

\begin{lemma} \label{lemzd}
For a nonzero $f$, the ring $S$ is a domain if and only if $R$ is a domain.
\end{lemma}

\begin{proof}
The direct implication follows from the inclusion $R=S_0\subseteq S$. Let us prove the converse.

Assume that $g_1g_2=0$ for some nonzero $g_1,g_2\in S$. Denote by $L(g)$ the nonzero homogeneous component of maximal degree of a nonzero element $g\in S$. Then $L(g_1)L(g_2)=0$.  On the other hand, we have $L(g_1)L(g_2)=abf^du^pv^q,$ where $a,b\in R\setminus\{0\}$ and $p,q,d$ are non-negative integers with $pq=0$. This leads to a contradiction.
\end{proof}

If the element $f$ is invertible in $R$, then replacing $v$ by $fv$ we come to the relation $uv=1$. This shows that $S$ is the ring of Laurent polynomials over $R$.

\smallskip

From now on we assume that $R$ is a domain and $f$ is nonzero non-invertible in $R$. Denote by $A^{\times}$ the set of invertible elements of a domain $A$.

\begin{lemma} \label{nl1}
We have $R^{\times}=S^{\times}$. In particular, the elements $u$ and $v$ are not invertible. 
\end{lemma}

\begin{proof}
Let $g\in S$ and $gh=1$ for some $h\in S$. Since $1$ is homogeneous, the elements $g$ and $h$ are homogeneous as well.  So we may assume that $g=u^da$ and $h=v^db$, where $a,b\in R$. If $d=0$ then $g\in R^{\times}$. If $d>0$ then $gh=f^dab=1$.  This contradicts the condition that $f$ is not invertible.
\end{proof}

Let us recall that a nonzero element $r$ in a domain $R$ is irreducible if $r$ is not invertible and if whenever $r=r_1r_2$ with $r_1,r_2\in R$, then one of $r_1$ and $r_2$ is invertible.

\begin{lemma} \label{nl2}
The elements $u$ and $v$ are irreducible in $S$.
\end{lemma}

\begin{proof}
If $u=gh$ then $g$ and $h$ are homogeneous. We may assume that $g=u^{d+1}a$ and $h=v^db$ for $a,b\in R$. Then $u=gh=uf^dab$. This implies $f^dab=1$. Since $f$ is not invertible, we have $d=0$ and $b$ is invertible. So the element $h$ is invertible. We conclude that $u$ is irreducible. The same arguments work for $v$.
\end{proof}

\begin{lemma} \label{nl3}
Let $a\in R$. Then
\begin{enumerate}
\item[\normalfont{(i)}]
if $a$ is irreducible in $S$, then $a$ is irreducible in $R$;


\item[\normalfont{(ii)}]
if $a$ is irreducible in $R$ and not associated with $f$, then $a$ is irreducible in $S$.
\end{enumerate}
\end{lemma}

\begin{proof}
We start with~(i). If $a=cd$ with $c,d\in R$, then we may assume that $c$ is invertible in~$S$. By Lemma~\ref{nl1}, the element $c$ is invertible in $R$, and so $a$ is irreducible in~$R$.

Now we come to~(ii). If $a=gh$ and $g,h\in S$, then $g=u^db$, $h=v^dc$ for some $b,c\in R$. If $d=0$ then either $g$ or $h$ is invertible by assumption. Assume that $d>0$.  Then $a=gh=f^dbc$. Since $a$ is irreducible in $R$, we have $d=1$ and $bc$ is invertible. Then $a$ is associated with $f$, a contradiction.
\end{proof}

\begin{lemma} \label{nl4}
An element $g=b_lv^l+\ldots+b_1v+b_0$, $b_i\in R$ is divisible by $u$ if and only if all the $b_i$ are divisible by $f$.
\end{lemma} 

\begin{proof}
If every $b_i$ is divisible by $f=uv$, then $g$ is divisible by $u$. Conversely, if $g=uh$, then $h=c_{l+1}v^{l+1}+\ldots+c_2v^2+c_1v$, where $c_i\in R$, and so
\[
g = u h = c_{l + 1} f v^l + \ldots + c_2 f v + c_1 f.
\qedhere
\]
\end{proof}

Recall that a nonzero non-invertible element $r$ of a domain $R$ is prime if whenever $r$ divides a product $ab$, then $r$ divides either $a$ or $b$. In other words, the principal ideal $(r)$ is prime in~$R$. Note that any prime element is irreducible.

\begin{proposition} \label{np1}
Let $R$ be a domain and $f \in R$ be a nonzero non-invertible element. Denote $S = R[u, v] / (u v - f)$. Then the element $u$ is prime in $S$ if and only if the element $v$ is prime in $S$ if and only if the element $f$ is prime in $R$.
\end{proposition}

\begin{proof}
By symmetry, it suffices to prove that $u$ is prime if and only if $f$ is prime in $R$.  Assume that $f$ is prime in $R$. Take two elements $g,h\in S$ not divisible by $u$. We may assume that
$$
g=b_lv^l+\ldots+b_1v+b_0
\quad \text{and} \quad
h=c_mv^m+\ldots+c_1v+c_0,
$$
where all $b_i,c_j\in R$ are not divisible by $f$. The lowest homogeneous component of the product $gh$ equals $b_lc_mv^{l+m}$. By assumptions, the element $b_lc_m$ is not divisible by $f$. By Lemma~\ref{nl4}, the element $gh$ is not divisible by $u$. We conclude that $u$ is prime.

Now assume that $u$ is prime. If $f$ divides $ab$ for some $a,b\in R$, then $uv$ divides $ab$ in $S$. Since $u$ is prime, we may assume that $u$ divides $a$. Then $a=uc$ with $c=vd$ and $d\in R$. This implies $a=uvd=fd$, so $f$ divides $a$ in $R$, and the element $f$ is prime in $R$.
\end{proof}

The following result shows that the construction of suspension preserves normality.

\begin{proposition} \label{norm}
Let $R$ be a domain and $f \in R$ be a nonzero non-invertible element. Then the ring $S = R[u, v] / (u v - f)$ is integrally closed if and only if the ring $R$ is integrally closed.
\end{proposition}

\begin{proof}
Assume that $S$ is integrally closed. Consider $a / b \in \Frac(R)$ integral over $R$, where $a, b \in R$ and $b \ne 0$. Since $S$ is integrally closed, we have $a / b = g$ for some $g \in S$. From $a = b g$ it follows that $g \in R$. Therefore, $a / b \in R$ and $R$ is integrally closed.

Conversely, assume that $R$ is integrally closed. Consider $\phi = p / q \in \Frac(S)$ integral over~$S$, where $p, q \in S$ and $q \ne 0$. For all $j \in \ZZ_{\ge 0 }$ we have
$$
\phi = \frac{p}{q} = \frac{u^j p}{u^j q},
$$
hence we may assume that $p, q \in R[u]$. Consider an integral equation
$$
\phi^m + s_1 \phi^{m - 1} + \ldots + s_m = 0,
$$
where $s_i \in S$ for all $i = 1, \dots, m$. Multiplying this equation by $u^{m k}$ for large enough $k \in \ZZ_{\ge 0}$, we see that $u^k \phi \in \Frac(R[u])$ is integral over $R[u]$. But $R[u]$ is integrally closed, hence $u^k \phi = g \in R[u]$. By a symmetric argument for $v$ there exists $l \in \ZZ_{\ge 0}$ such that $v^l \phi = h \in R[v]$. We obtain
$$
\phi = \frac{g}{u^k} = \frac{h}{v^l}.
$$
For $d = \max(k, l)$ we can rewrite these equations as
$$
\phi = \frac{u^{d - k} g}{u^d} = \frac{v^{d - k} h}{v^d}.
$$
Clearly, $u^{d - k} g \in R[u]$ and $v^{d - k} h \in R[v]$. So it remains to prove the following lemma.

\begin{lemma} \label{le}
Consider an element $\phi \in \Frac(S)$. Assume that there exist $d \in \ZZ_{\ge 0}$, $g \in R[u]$, and $h \in R[v]$ such that
$$
\phi = \frac{g}{u^d} =\frac{h}{v^d}.
$$
Then $\phi \in S$.
\end{lemma}

\begin{proof}
We argue by induction on $d$. The case $d = 0$ is trivial. Assume that the assertion is true for some $d \in \ZZ_{\ge 0}$. Consider $\phi \in \Frac(S)$, $g \in R[u]$, and $h \in R[v]$ such that
$$
\phi = \frac{g}{u^{d + 1}} =\frac{h}{v^{d + 1}}.
$$
Write
$$
g = a + u g' \quad \text{and} \quad h = b + v h',
$$
where $a, b \in R$, $g' \in R[u]$, and $h' \in R[v]$. Denote
$$
g' = \sum_i a_i u^i \quad
\text{and} \quad
h' = \sum_i b_i v^i,
$$
where all $a_i, b_i$ are elements of $R$. Comparing coefficients of $v^{d + 1}$ and $u^{d + 1}$ in the equality $v^{d + 1} g = u^{d + 1} h$ we conclude that
$$
a = b_{2 d + 1} f^{d + 1} \quad
\text{and} \quad
b = a_{2 d + 1} f^{d + 1}.
$$
Therefore,
$$
\frac{g}{u^{d + 1}} = \frac{a}{u^{d + 1}} + \frac{g'}{u^d} =
b_{2 d + 1} v^{d + 1} + \frac{g'}{u^d} \quad
\text{and} \quad
\frac{h}{v^{d + 1}} = \frac{b}{v^{d + 1}} + \frac{h'}{v^d} =
a_{2 d + 1} u^{d + 1} + \frac{h'}{v^d}.
$$
Denote $\phi' = \phi - b_{2 d + 1} v^{d + 1} - a_{2 d + 1} u^{d + 1}$. Then
$$
\phi' = \frac{g'}{u^d} - a_{2 d + 1} u^{d + 1} =
\frac{h'}{v^d} - b_{2 d + 1} v^{d + 1}.
$$
Clearly, the elements
$$
\frac{g'}{u^d} - a_{2 d + 1} u^{d + 1} \quad 
\text{and} \quad
\frac{h'}{v^d} - b_{2 d + 1} v^{d + 1}
$$
can be written in the form $\frac{g''}{u^d}$ and $\frac{h''}{v^d}$ respectively, where $g'' \in R[u]$ and $h'' \in R[v]$. By the induction hypothesis $\phi' \in S$. Since $- b_{2 d + 1} v^{d + 1} - a_{2 d + 1} u^{d + 1} \in S$, we conclude that $\phi \in S$.
\end{proof}
This completes the proof of Proposition~\ref{norm}.
\end{proof}


\section{A criterion of factoriality}
\label{sec2}


Now we come to factoriality. A domain $R$ is a unique factorization domain (UFD) if nonzero non-invertible elements of $R$ can be factored uniquely into irreducible elements up to renumbering and invertible factors. A domain $R$ is a UFD if and only if any irreducible element in $R$ is prime and any ascending chain of principal ideals in $R$ terminates, see e.g. \cite[Section~0.2]{Eis}. 

\begin{theorem} \label{tmain}
Let $R$ be a domain and $f \in R$ be a nonzero non-invertible element. Then the domain
$$
S = R[u, v] / (u v - f)
$$
is a UFD  if and only if $R$ is a UFD and $f$ is prime in $R$. 
\end{theorem}

\begin{proof}
Assume that $S$ is factorial. By Lemma~\ref{nl2}, the element $u$ is irreducible. By factoriality of $S$, the element $u$ is prime. Proposition~\ref{np1} implies that
$f$ is prime in $R$.

Let us show that any ascending chain of principal ideals in $R$ terminates. Consider such a chain $Ra_1\subseteq Ra_2\subseteq\ldots$.  Since $S$ is factorial, the corresponding chain $Sa_1\subseteq Sa_2\subseteq\ldots$ of principal ideals in $S$ terminates. But if the principal ideals $Sa_i$ and $Sa_{i+1}$ coincide, then the elements $a_i$ and $a_{i+1}$ are associated in $S$. By Lemma~\ref{nl1}, the elements $a_i$ and $a_{i+1}$ are associated in $R$ as well, and the original chain in $R$ terminates.

It remains to check that any irreducible element $a\in R$ is prime in $R$. If $a$ is associated with $f$, then the assertion follows from Proposition~\ref{np1}.  Let us assume that $a$ is not associated with $f$. By Lemma~\ref{nl3}~(ii), $a$ is irreducible in $S$, and so it is prime in $S$. If $a$ divides $bc$ with $b,c\in R$, then we may assume that $a$ divides $b$ in $S$. Then $a$ divides $b$ in $R$, and so $a$ is prime in $R$.

\medskip

Assume now that $R$ is factorial and $f$ is prime in $R$. Then the polynomial ring $R[u,v]$ is factorial as well. So, any ascending chain of principal ideals in $R[u,v]$ and in its factor ring $S$ terminates.

It remains to prove that any irreducible $g\in S$ is prime. If $u$ divides $g$, then $g$ is associated with $u$. In this case $g$ is prime by Proposition~\ref{np1}. Further we assume that $u$ does not divide~$g$.

\smallskip

Take the smallest non-negative integer $l$ such that $u^lg\in R[u]$. 

\smallskip

{\it Step 1}.\ Let us show that the element $u^lg$ is prime in $R[u]$. Since $R[u]$ is UFD, it suffices to prove that $u^lg$ is irreducible in $R[u]$.

Let $u^lg=p_1p_2$ with $p_1,p_2\in R[u]$. By Proposition~\ref{np1}, the element $u$ is prime in $S$. Since $g$ is irreducible in $S$, after dividing the right-hand side by $u^l$ we cannot obtain a product of two non-invertible elements. So we may assume that $p_1$ is associated in $S$ with $u^t$ for some $0\le t\le l$. This means that $p_1=\lambda u^t$, where $\lambda\in S^{\times}=R^{\times}$. If $t>0$ then
$
u^{l - t} g = \lambda p_2 \in R[u],
$
a~contradiction with minimality of $l$. So $t=0$ and $p_1\in R^{\times}$. This shows that $u^lg$ is prime in~$R[u]$.

\smallskip

{\it Step 2}.\ Let us prove that $g$ is prime in $S$. Let $h_1h_2=gs$ with $h_1,h_2,s\in S$. Multiplying both sides by a suitable power of $u$, we obtain $h_1'h_2'=u^lgs'$, where $h_1',h_2',s'\in R[u]$. Since the element $u^lg$ is prime in $R[u]$, we may assume that $h_1'=u^lgp$ with $p\in R[u]$. By construction, we have $h_1'=u^rh_1$ with $r\in\ZZ_{\ge 0}$.

\smallskip

If $r\le l$ then $h_1=u^{l-r}gp$, so $g$ divides $h_1$ in $S$.

\smallskip

If $r>l$ then $u^{r-l}h_1=gp$. Since $u$ does not divide $g$ and $u$ is prime in $S$, we conclude that $u^{r-l}$ divides $p$ in $S$. So $h_1=\frac{p}{u^{r-l}}g$, and $g$ divides $h_1$ in $S$.

\smallskip

This proves that $g$ is prime in $S$. The proof of Theorem~\ref{tmain} is completed.
\end{proof}

\begin{remark}
Let $A=\oplus_{i\in\ZZ} A_i$ be a $\ZZ$-graded domain. We say that a nonzero non-invertible homogeneous element $a\in S$ is $\ZZ$-prime if $a$ divides a product $gh$ of homogeneous elements if and only if either $a$ divides $g$ or $a$ divides $h$. Further, $A$ is said to be factorially graded if any nonzero non-invertible homogeneous element in $A$ is a product of $\ZZ$-primes.

Under some restrictions on the domain $A$ it is known that $A$ is  factorially graded if and only if $A$ is UFD; see~\cite{An} or \cite[Theorem~3.4.1.11~(ii)]{ADHL}. This result can simplify the proof of Theorem~\ref{tmain} at the expense of loss of generality.
\end{remark}


\section{Homogeneous varieties vs homogeneous spaces}
\label{sec3}


Let us start with a direct corollary of Theorem~\ref{tmain} and the criterion of smoothness of suspensions~\cite[Corollary~1]{AZa}.

\begin{theorem} \label{tsec}
Let $Y$ be a flexible affine variety. Then a suspension $\Susp(Y,f)$ is a smooth flexible affine factorial variety if and only if $Y$ is factorial and smooth, the function $f$ is a prime element in $\KK[Y]$, and $\{f=0\}$ is a smooth subvariety in $Y$.
\end{theorem}

Theorem~\ref{tsec} provides an iterative procedure for constructing a large number of homogeneous affine factorial varieties in each dimension starting from $3$.

\smallskip

From now until the end of this section we assume that the ground field $\KK$ is the field of complex numbers $\CC$.

\begin{proposition} \label{psc}
Let an irreducible normal affine variety $Y$ be simply connected and $f$ be a non-invertible nonzero regular function on $Y$. Then the suspension $\Susp(Y,f)$ is simply connected.
\end{proposition}

\begin{proof}
By~\cite[Theorem~7.1]{Sha}, any irreducible complex algebraic variety is connected in the complex topology.  The variety $X=\Susp(Y,f)$  is given by $uv=f(y)$. Consider the principal open subset $X_u$ defined by $u\ne 0$. The variety $X_u$ is isomorphic to $\AA^1_*\times Y$, $(u,v,y)\mapsto (u,y)$, and its fundamental group is generated by the loop $(l(t),y_0)$, $t\in S^1$, where $l(t)$ is a circle around the punctured point on $\AA^1_*$, and $y_0\in Y$. One may choose $y_0$ such that $f(y_0)=0$.

By Proposition~\ref{norm}, the variety $X$ is normal, and by~\cite[Proposition~2.10.1]{Kol}, the inclusion $X_u\subseteq X$ induces  a surjection of fundamental groups. So the fundamental group of $X$ is generated by the class of the image of the loop $(l(t),y_0)$. This loop maps to $(l(t), 0, y_0)$, since $v=f(y_0)/u=0$. So it is contained in the subvariety $(u, 0, y_0)$, which is isomorphic to $\AA^1$, hence this loop is contractible. We conclude that $X$ is simply connected.
\end{proof}

More information on topology of suspensions over $Y=\CC^n$ can be found~\cite[Section~4]{KZ}. Let us formulate some of these results that we use below.

Let $X$ be a topological space. First we recall relations between unreduced and reduced singular homology $H_k(X)$ and $\tilde{H}_k(X)$, $k \in \ZZ_{\ge 0}$, with coefficients in $\ZZ$. We have
$$
H_k(X) \cong \tilde{H}_k(X)  \ \text{for} \  k \ge 1
$$
and
$$
H_0(X) \cong \tilde{H}_0(X) \oplus \bb{Z}.
$$
Denote by $\pi_0(X)$ the set of path-connected components of $X$. Then
$$
H_0(X) \cong \bb{Z}[\pi_0(X)],
$$
where $\bb{Z}[A]$ for a set $A$ denotes the free abelian group with basis $A$.

Let $(X, x)$ be a pointed topological space. For each $k \in \ZZ_{\ge 0}$ there is a Hurewicz homomorphism
$$
\Phi_k \colon \pi_k(X, x) \to H_k(X).
$$
The Hurewicz Theorem states the following. If $X$ is $(k - 1)$-connected for some $k \in \ZZ_{>0}$, then $\Phi_k$ induces an isomorphism
$$
\pi_k^\mathrm{ab}(X, x) \cong H_k(X),
$$
where $\pi_k^\mathrm{ab}(X, x)$ is the abelianization of the homotopy group $\pi_k(X, x)$. Note that for $k \ge 2$, since the group $\pi_k(X, x)$ is abelian, we obtain an isomorphism $\pi_k(X, x) \cong H_k(X)$.

\smallskip

Consider a suspension $X = \Susp(\CC^n, f)$ for some non-constant $f \in \KK[\CC^n]$. Assume that the hypersurface
$$
X_0 = \Spec \KK[\CC^n] / (f)
$$
in $\CC^n$ is reduced, irreducible and smooth. By~\cite[Proposition~4.1]{KZ}, there are isomorphisms
\begin{equation} \tag{$*$} \label{iso}
\tilde{H}_*(X) \cong \tilde{H}_{* - 2}(X_0).
\end{equation}
We have
$$
\tilde{H}_0(X) = \tilde{H}_0(X_0)=0.
$$
Proposition~\ref{psc} implies the equality $\pi_1(X) = 0$, hence
$$
H_1(X) \cong \pi_1^\mathrm{ab}(X) = 0.
$$
From~(\ref{iso}) we obtain
$$
H_2(X) = 0  \quad \text{and}  \quad H_3(X) \cong H_1(X_0).
$$
In particular, $\pi_2(X) = 0$ (so $X$ is actually always $2$-connected) and
$$
\pi_3(X) \cong H_3(X) \cong H_1(X_0) \cong
\pi_1^\mathrm{ab}(X_0).
$$

We are ready to come to concrete examples and to construct a homogeneous affine factorial variety that is not a homogeneous space.

The only one-dimensional affine factorial variety $Y$ without non-constant invertible functions is the affine line $\CC^1$. The only irreducible polynomials $f$ on $\CC^1$ are linear functions. So the only smooth affine factorial surface $X$ that can be obtained as $\Susp(Y, f)$ with flexible $Y$ is the affine plane $\CC^2$. The plane is homogeneous with respect to several actions of linear algebraic groups.

The situation with flexible affine factorial 3-folds is more diverse.

\begin{theorem} \label{texa}
The hypersurface $x^2y-xy-uv+1=0$ in $\CC^4$ is a homogeneous affine factorial 3-fold that is not a homogeneous space of an algebraic group.
\end{theorem}

\begin{proof}
The hypersurface $x^2y-xy-uv+1=0$ is the suspension
$$
X = \Susp(\CC^2, \ (x - 1) x y + 1).
$$
By Theorem~\ref{tsec},  $X$  is a smooth flexible affine factorial variety. In particular, $X$ is homogeneous.

It follows from~\cite[Theorem~2 and Proposition~19]{Po-2} that there are precisely three pairwise non-isomorphic affine factorial 3-folds  without non-constant invertible functions that are homogeneous spaces of algebraic groups: they are $\CC^3$, $\SL(2)$ and $\SL(2)/2I$, where $2I$ is the binary icosahedral group. The last group has order~$120$ and is the universal perfect central extension of the simple group $A_5$.

By Proposition~\ref{psc}, the hypersurface $X$ is simply connected, so it is not isomorphic to $\SL(2)/2I$.  Moreover, we know that $\pi_2(X)=0$ and $\pi_3(X)\cong \pi_1^\mathrm{ab}(X_0)$, where $X_0$ is the curve $\{(x-1)xy+1=0\}$ in $\CC^2$. This curve is isomorphic to $\CC\setminus \{0,1\}$, so the group  $\pi_1^\mathrm{ab}(X_0)$ is the lattice $\ZZ^2$. We conclude that $X$ is not isomorphic to $\CC^3$, since the last variety is contractible. Moreover, the variety $\SL(2)$ is homotopic to the real $3$-sphere, so the group $\pi_3(\SL(2))$ is isomorphic to $\ZZ$. This proves that $X$ is not a homogeneous space of an algebraic group.
\end{proof}

It is an important problem to find algebraic invariants that would allow to distinguish between non-isomorphic smooth flexible affine factorial varieties of the same dimension.


\section{The divisor class group}
\label{sec4}


Let us compute the divisor class group $\Cl(X)$ of a suspension $X = \Susp(Y, f)$ over a normal affine variety $Y$ for an arbitrary non-constant function~$f$. This gives an alternative proof of Theorem~\ref{tmain} in geometric settings.

We begin with some notations. Let $X$ be a variety and let $g \in \bb{K}[X]$. Denote by $X_g$ the principal open subset of $X$ defined by $g \ne 0$. Let $\bb{V}_X(g)$ be the closed subvariety of $X$ given by the equation $g = 0$. In other words, $\bb{V}_X(g) = X \setminus X_g$.

Let $X$ be an irreducible normal variety and $g \in \bb{K}(X)^{\times}$. Denote by $\WDiv(X)$ the group of Weil divisors on $X$ and by $\div_X(g) \in \WDiv(X)$ the principal divisor of $g$. For a prime divisor $D$ on $X$, denote by $\nu_D(g)$ the order of zero/pole of $g$ along $D$. In other words,
$$
\div_X(g) =
\sum_{\substack{D \subset X \\ \text{prime}}} \nu_D(g) \cdot D.
$$

Let $U \subseteq X$ be an open subset. There is a surjective homomorphism $\phi \colon \Cl(X) \to \Cl(U)$ induced by the map
$$
\bar{\phi} \colon \WDiv(X) \to \WDiv(U), \quad
D \mapsto D \cap U,
$$
where $D \in \WDiv(X)$ is a prime divisor and if $D \subseteq X \setminus U$, then we mean $D \cap U = 0$. If $D_1, \dots, D_s$, $s \ge 0$, is the set of prime divisors on $X$ contained in $X \setminus U$, then there is an exact sequence
$$
\begin{tikzcd}
\bb{Z}^s \ar{r}{\psi} & \Cl(X) \ar{r}{\phi} & \Cl(U) \ar{r} & 0
\end{tikzcd}
,
$$
where $\psi$ is the map
$$
(a_1, \dots, a_s) \mapsto a_1 [D_1] + \ldots + a_s [D_s].
$$
We say that the maps $\phi$ and $\psi$ are induced by the open embedding $U \subseteq X$.

Note also that there is an isomorphism
$$
\theta_X \colon \Cl(X) \cong \Cl(\bb{A}^1_* \times X)
$$
induced by the map
$$
\bar{\theta}_X \colon
\WDiv(X) \to \WDiv(\bb{A}^1_* \times X), \quad
D \mapsto \bb{A}^1_* \times D.
$$

From now on in this section we assume that $Y$ is an irreducible normal affine variety, $f \in \bb{K}[Y] \setminus \{0\}$, and $X = \Susp(Y, f) \subset \bb{A}^2_{u, v} \times Y$. Proposition~\ref{norm} shows that $X$ is normal as well.

Note that if $u \ne 0$, then $v = \frac{f}{u}$ in $X$, hence $X_u \cong \bb{A}^1_* \times Y$. Conversely, if $u = 0$, then $f = 0$, hence $\bb{V}_X(u) \cong \bb{A}^1 \times \bb{V}_Y(f)$. By symmetry, there are similar isomorphisms for $v$.

Denote by $E_1, \dots, E_s$, $s \ge 0$, the set of prime divisors on $Y$ contained in $\bb{V}_Y(f)$. If we denote by $D_i$ the subvariety $\bb{A}^1 \times E_i \subseteq \bb{V}_X(u)$, then $D_1, \dots, D_s$ is the set of prime divisors on $X$ contained in $\bb{V}_X(u)$. Let
$$
\div_Y(f) = a_1 E_1 + \ldots + a_s E_s \in \WDiv(Y),
$$
where each $a_i$ is a positive integer.

\begin{lemma} \label{PDu.le}
The following equality holds:
$$
\div_X(u) = a_1 D_1 + \ldots + a_s D_s \in \WDiv(X).
$$
\end{lemma}

\begin{proof}
Fix $1 \le i \le s$. Clearly, the divisor $\div_X(u)$ is supported on $D_1\cup\ldots\cup D_s$, so it suffices to show that $\nu_{D_i}(u) = a_i$. We have $\nu_{D_i}(v) = 0$, since $\bb{V}_X(v)$ intersects $\bb{V}_X(u)$ in a subvariety of codimension $2$ in $X$. Therefore, we obtain
\[
a_i = \nu_{D_i}(f) = \nu_{D_i}(u v) =
\nu_{D_i}(u) + \nu_{D_i}(v) = \nu_{D_i}(u).
\qedhere
\]
\end{proof}

\begin{proposition} \label{seq}
There is an exact sequence
$$
\begin{tikzcd}
0 \ar{r} & \bb{Z} \ar{r}{\xi} & \bb{Z}^s \ar{r}{\psi} &
\Cl(X) \ar{r}{\phi'} & \Cl(Y) \ar{r} & 0
\end{tikzcd}
,
$$
where $\phi' = \theta_{Y}^{-1} \circ \phi$, the maps $\phi$ and $\psi$ are induced by the open embedding $X_u \subseteq X$, and $\xi$ is defined by $1 \mapsto (a_1, \dots, a_s)$.
\end{proposition}

\begin{proof}
Clearly, the sequences
$$
\begin{tikzcd}
0 \ar{r} & \bb{Z} \ar{r}{\xi} & \bb{Z}^s
\end{tikzcd}
\quad \text{and} \quad
\begin{tikzcd}
\bb{Z}^s \ar{r}{\psi} & \Cl(X) \ar{r}{\phi} &
\Cl(X_u) \ar{r} & 0
\end{tikzcd}
$$
are exact, so it suffices to show that $\Im \xi = \Ker \psi$. By Lemma~\ref{PDu.le} we have $\Im \xi \subseteq \Ker \psi$. Consider $(b_1, \dots, b_s) \in \Ker \psi$, that is,
$$
b_1 D_1 + \ldots + b_s D_s = \div_X(g)
$$
for some $g \in \bb{K}(X)^{\times}$. It follows that $g$ is regular and invertible on $X \setminus \bb{V}_X(u) = X_u \cong \bb{A}^1_* \times Y$. By~a result of Rosenlicht, see \cite[Proposition~1.1]{KKV}, there exist $p \in \bb{K}[\bb{A}^1_*]^\times$ and $q \in \bb{K}[Y]^\times$ such that $g = p q$. We may assume that $p = u^k$ for some $k \in \bb{Z}$ and $q$ is a regular invertible function on~$X$. Therefore,
$$
\div_X(g) = \div_X(p) + \div_X(q) = \div_X(u^k) + 0 =
k \cdot (a_1 D_1 + \ldots + a_s D_s),
$$
so $(b_1, \dots, b_s) \in \Im \xi$.
\end{proof}

\begin{corollary}
Let $Y$ be an irreducible normal affine variety and $f$ be a nonzero non-invertible element in $\KK[Y]$. Then the variety $X=\Susp(Y,f)$ is factorial if and only if $Y$ is factorial and the element $f$ is prime.
\end{corollary}

\begin{proof}
By Proposition~\ref{seq}, we have $\Cl(X)=0$ if and only if $\Cl(Y)=0$ and the map ${\xi\colon \ZZ\to\ZZ^s}$ is an isomorphism. The first condition means that $Y$ is factorial, while the second means that $f$ is prime.
\end{proof}

The next corollary generalizes~\cite[Theorem~4(a)]{AZa}.

\begin{corollary} \label{pdcg}
Let $Y$ be an affine factorial variety. Consider a function $f\in\KK[Y]$ with ${f=p_1^{a_1}\ldots p_s^{a_s}}$, where $p_i$ are pairwise distinct primes  in $\KK[Y]$ and $a_i\in\ZZ_{>0}$. Let $X =$ \linebreak
$= \Susp(Y,f)$. Then $\Cl(X) \cong \ZZ^s/\langle \omega \rangle$, where $\omega=(a_1,\ldots,a_s)$.
\end{corollary}

The last corollary is straightforward.

\begin{corollary}
In the notation of Corollary~\ref{pdcg},  the group $\Cl(X)$ has no torsion if and only if the vector $\omega$ is primitive or, equivalently, the element $f$ is not a proper power.
\end{corollary}


\section{A remark on suspensions}
\label{sec5}


Consider a suspension $X=\Susp(Y,f)$ over a normal affine variety $Y$ corresponding to a non-constant regular function $f$. Recall that $X$ is given in the direct product $\AA^2_{u,v}\times Y$ by the equation $uv=f$.  Denote by $\GG_m$ the multiplicative group of the ground field $\KK$.

The variety $X$ carries a $\GG_m$-action given by
$$
t\cdot (u, v, y)=(tu, t^{-1}v, y).
$$
The algebra of regular invariants $\KK[X]^{\GG_m}$ coincides with the subalgebra $\KK[Y]\subseteq\KK[X]$ and the quotient morphism
$$
\pi\colon X\to X/\!/\GG_m:=\Spec \KK[X]^{\GG_m}=Y
$$
is the projection $(u,v,y)\mapsto y$. If $D$ is the principal divisor $\divv_Y(f)$ on $Y$, then the fiber $\pi^{-1}(y)$ over $y\in Y\setminus \Supp(D)$ is a hyperbola $uv=c$, $c\ne 0$,  that is a closed $\GG_m$-orbit on $X$. In turn, the fiber $\pi^{-1}(y)$ over $y\in\Supp(D)$ is the union $uv=0$ of two lines that form three $\GG_m$-orbits.

\smallskip

One may expect that this way we obtain a characterization of suspensions. Namely, let $X$ be an affine variety with a $\GG_m$-action and $\pi\colon X\to Y:=\Spec\KK[X]^{\GG_m}$ be the quotient morphism. Assume that there is a principal divisor $D=\divv_Y(f)$ on $Y$ such that the fiber $\pi^{-1}(y)$ with $y\in Y\setminus \Supp(D)$ is one $\GG_m$-orbit and $\pi^{-1}(y)$ with $y\in\Supp(D)$ is a union of two intersecting lines. Then $X$ is the suspension $\Susp(Y,f)$ over $Y$.

Let us show that this is not the case. Namely, let us consider the affine 3-fold $X$ given as
$$
\bb{V}(u_1v-y_1y_2, \ u_2v-(y_1+1)y_2, \ u_1(y_1+1)-u_2y_1)\subseteq\AA^5
$$
with the $\GG_m$-action
$$
t\cdot (u_1, u_2, v, y_1, y_2)=(tu_1, tu_2, t^{-1}v, y_1, y_2).
$$
The quotient morphism $\pi$ sends the point $(u_1,u_2,v,y_1,y_2)$ to $(y_1,y_2)$. Hence $Y=\AA^2$ and with the divisor $D=\divv_Y(y_2)$ we have the desired fiber structure. At the same time, the $\GG_m$-action defines a $\ZZ$-grading on the algebra $A:=\KK[X]$ and the homogeneous component
$$
A_1=u_1\KK[Y]+u_2\KK[Y]
$$
is not a cyclic $\KK[Y]$-module. This proves that this $\GG_m$-action on $X$ cannot be induced by a suspension.



\end{document}